\documentclass{elsarticle}

\usepackage{amssymb, latexsym, comment, amsmath, amsthm, upgreek, epsf, epsfig, color, verbatim, caption, tabularx, xfrac, float}
\usepackage{amsfonts}
\usepackage{booktabs}
\usepackage{siunitx}
\usepackage{footnote}
\usepackage{url}

\usepackage{array}
\newcolumntype{L}[1]{>{\raggedright\let\newline\\\arraybackslash\hspace{0pt}}m{#1}}
\newcolumntype{C}[1]{>{\centering\let\newline\\\arraybackslash\hspace{0pt}}m{#1}}
\newcolumntype{R}[1]{>{\raggedleft\let\newline\\\arraybackslash\hspace{0pt}}m{#1}}

\makeatletter
\def\ps@pprintTitle{%
\let\@oddhead\@empty
\let\@evenhead\@empty
\def\@oddfoot{\centerline{\thepage}}%
\let\@evenfoot\@oddfoot}
\makeatother

\newtheorem{theorem}{Theorem}[section]
\newtheorem{lemma}[theorem]{Lemma}
\newtheorem{proposition}[theorem]{Proposition}
\newtheorem{cor}[theorem]{Corollary}
\newtheorem{definition}[theorem]{Definition}

\begin{document}

\begin{frontmatter}
\title{Minimizing geodesic nets and critical points of distance}
\author{Ian M Adelstein} 
\address{Department of Mathematics, Yale University \\ New Haven, CT 06520 United States}
\begin{abstract} We establish a relationship between geodesic nets and critical points of the distance function. We bound the number of balanced points for certain minimizing geodesic nets on manifolds homeomorphic to the $n$-sphere. This result is used to give conditions under which a minimizing geodesic flower degenerates into a simple closed geodesic. 
\end{abstract}
\end{frontmatter}

\section{Introduction}

A net is a (finite) multigraph embedded into a Riemannian manifold. The vertices of a net are partitioned into \emph{balanced} and  \emph{boundary} vertices. A number of adjectives are often added to the definition of a net, c.f.~Nabutovsky and Rotman \cite{NR2007}. A \emph{geodesic} net is a net such that each edge is a geodesic, and a \emph{stationary} geodesic net is a geodesic net such that at each balanced vertex, the sum of the unit vectors tangent to the incident edges equals zero. As such, stationary geodesic nets are critical points of the length/energy functional on the space of graphs with prescribed boundary vertex set embedded in a Riemannian manifold. In this paper we work with a new, more restrictive class of nets:

\begin{definition} A \emph{minimizing geodesic net} is a stationary geodesic net such that each edge realizes the distance between its vertices. 
\end{definition}

We provide some examples. In flat Euclidean space, every stationary geodesic net is a minimizing geodesic net, as geodesics in Euclidean space minimize between any pair of their points. There are open questions about stationary geodesic nets even in the Euclidean setting, c.f.~Nabutovsky and Parsch \cite{NP2019}.

Hass and Morgan \cite{morgan} provide one of the only known existence results on geodesic nets, demonstrating that convex metrics on the 2-sphere sufficiently close to the round metric necessarily contain a stationary geodesic net modeled on the theta-graph. The theta-graph is a graph with three edges between two vertices. Such a theta-graph can be realized as a minimizing geodesic net on the round 2-sphere by taking the poles as the vertices, and three great semicircles of longitude as the edges, each meeting at angle $2\pi/3$ at the poles (see Figure~\ref{theta}).

\begin{figure}[h]
  \includegraphics[scale=.4]{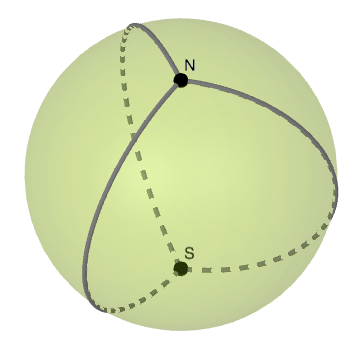}
  \centering
  \caption{A theta-graph on the sphere.}
  \label{theta}
\end{figure}

An important question in the study of geodesic nets is whether one can bound the number of balanced (stationary) vertices in terms of the other geometric properties of the net. Parsch \cite{fabian} has recently proved that a geodesic net with three unbalanced (boundary) vertices in the plane endowed with a Riemannian metric of non-positive curvature can admit at most one balanced vertex. Gromov \cite{Gro09} has conjectured that the number of balanced vertices of a geodesic net in the Euclidean plane can be bounded above in terms of the number of unbalanced vertices and the total imbalance, see also \cite[Conjecture 3.4.1]{NP2019}.

Our main result provides a bound on the number of balanced vertices on minimizing geodesic nets modeled on certain graphs embedded in Riemannian manifolds endowed with metrics of positive sectional curvature. Note that by the Grove-Shiohama diameter sphere theorem \cite{Gro} all manifolds considered in the following theorem are homeomorphic to the $n$-sphere. 

\begin{theorem}\label{main} Let $M$ be a complete Riemannian manifold with sectional curvature $k \geq 1$ and injectivity radius $> \pi/2$. Let $\Gamma$ be a star multigraph, i.e.~a multigraph with radius 1 and no cycles of length greater than 2 (see Figure~\ref{graph}). Then any minimizing geodesic net modeled on $\Gamma$ whose image contains a simple closed geodesic admits  at most two balanced vertices.
\end{theorem}

\begin{figure}[h]
  \includegraphics[scale=.3]{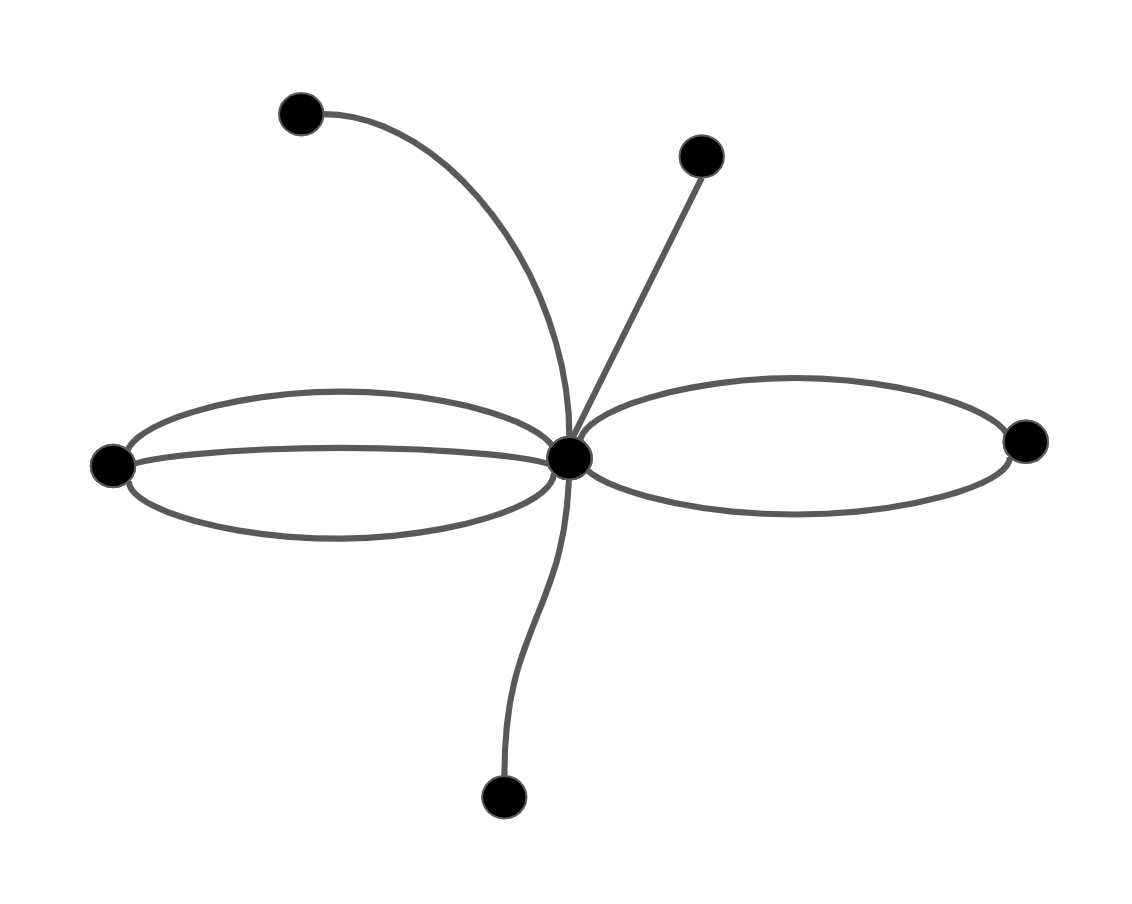}
  \centering
  \caption{A multigraph with radius 1 and no cycles of length greater than 2.}
  \label{graph}
\end{figure}

The stationarity condition at a balanced vertex of degree 2 implies that the two incident edges form a smooth geodesic. One can therefore always add or delete balanced vertices of degree 2 from a geodesic net without changing its nature. In the theorem above, the radius 1 condition implies that we can not add balanced vertices of degree 2, while the minimizing condition implies that we can not delete such vertices. 

The proof of the theorem relies on the notion of a Grove-Shiohama critical point of distance [Definition~\ref{GSdef2}]. For a minimizing geodesic net modeled on $\Gamma$, the central vertex (central is well defined, unless there are exactly two vertices, in which case either can be considered central) together with the halfway point along the assumed closed geodesic form a pair of mutual Grove-Shiohama critical points of distance. Under the assumed lower bounds on sectional curvature and injectivity radius, we show that mutual Grove-Shiohama critical points of distance are unique [Proposition~\ref{prop}]. It follows that the net cannot admit balanced vertices beyond this pair; full details are provided in Section~\ref{proof}. 

This theorem has a nice reformulation in the setting of minimizing geodesic flowers. A \emph{geodesic flower} is a bouquet of (finitely many) geodesic loops (petals) based at the same vertex and satisfying the stationary condition at the vertex. We define a \emph{minimizing} geodesic flower to be a geodesic flower that realizes the distance from the vertex to the halfway point along each of its petals. If a minimizing geodesic flower contains a simple closed geodesic, then the vertex together with the halfway point form a mutual Grove-Shiohama critical pair, and uniqueness [Proposition~\ref{prop}] yields the following:

\begin{cor}\label{flower}
Let $M$ be a complete Riemannian manifold with sectional curvature $k \geq 1$ and injectivity radius $> \pi/2$. Then any minimizing geodesic flower on $M$ whose image contains a simple closed geodesic must have exactly one petal, namely the simple closed geodesic. 
\end{cor}

Rotman \cite{rotman} (also \cite{NR2004} and \cite{NR2005}) has shown the existence of a constant $c(n)$ such that each closed Riemannian manifold $M^n$ admits a geodesic flower with length bounded above by $c(n)diam(M)$. An important open question due to Gromov \cite{Gro83} is whether there exists such a constant bounding the length of the shortest closed geodesic. In the non-simply connected case it is easy to show that the shortest closed geodesic has length bounded above by twice the diameter. The above corollary helps delineate between Rotman's bound and Gromov's question by giving conditions under which a geodesic flower degenerates into a simple closed geodesic on certain simply-connected manifolds. 

In the compact setting, any minimizing geodesic flower with $m$ petals must have length bounded above by $(2m)diam(M)$. In the non-compact setting such a universal upper bound does not exist, but we provide the following pointwise:

\begin{theorem}\label{bound}
Let $M$ be a complete non-compact Riemannian manifold with sectional curvature $k \geq 0$. Then for every $p \in M$ there exists a constant $R_p>0$ such that every minimizing geodesic flower with vertex $p$ and $m$ petals has length bounded above by $(2m)R_p$.  
\end{theorem}

Note that unlike in \cite{rotman}, our result does not provide an upper bound on the length of the shortest geodesic flower. We do not prove the existence of a minimizing geodesic flower; we merely provide a pointwise upper bound on the length should such a minimizing geodesic flower exist.

The paper proceeds as follows. In Section~\ref{proof} we introduce the relevant ideas from critical point theory. We prove the fact that mutual Grove-Shiohama critical points of distance are unique in our setting [Proposition~\ref{prop}]. We show how Theorem~\ref{main} and Corollary~\ref{flower} follow from this proposition. Finally, we show how Theorem~\ref{bound} follows from a lemma due to Gromov. For more on geodesic nets in this setting see Croke \cite{Croke}, Heppes \cite{Heppes}, Rotman \cite{Rotman07}, and \cite{Ade3}.

\section{Proof of the Theorems}\label{proof}

We first introduce the relevant ideas from critical point theory.

\begin{definition}\label{GSdef2} A \emph{Grove-Shiohama critical point} of $d_p \colon M \to \mathbb{R}$ is a point $q \in M$ such that for any $v \in T_qM$ there exists a minimizing geodesic from $q$ to $p$ with initial velocity vector $w \in T_qM$ such that $\measuredangle {(v, w)}  \leq \pi/2$.
\end{definition}

The original application of this critical point definition is the celebrated Grove-Shiohama diameter sphere theorem \cite{Gro}.

\begin{theorem}[Grove-Shiohama]\label{GSthm2} Let $M$ be a complete Riemannian manifold with sectional curvature $k \geq 1$ and diameter $> \pi/2$. Then $M$ is homeomorphic to the sphere. 
\end{theorem}

On Riemannian manifolds we know diameter $\geq$ injectivity radius, so that by the above result, all manifolds considered by Theorem~\ref{main} and Corollary~\ref{flower} are homeomorphic to the sphere. The theorem of Bonnet-Myers states that a complete Riemannian manifold with sectional curvature $k\geq 1$ has an upper bound of $\pi$ on the diameter, with Cheng \cite{cheng} proving that only the round sphere realizes this diameter upper bound. 

At the other extreme of the manifolds considered by Theorem~\ref{main} and Corollary~\ref{flower}, we note that the standard metric on $\mathbb{R}P^n$ has sectional curvature $k = 1$ and diameter = injectivity radius = $ \pi/2$. This projective space demonstrates that the bounds in Theorem~\ref{main} and Corollary~\ref{flower} are sharp, as $\mathbb{R}P^n$ admits minimizing geodesic flowers with any number of petals (any number of projected great circles through a central vertex). 

Theorem~\ref{main} and Corollary~\ref{flower} rely on the following proposition. Here we say that a pair $p,q \in M$ are \emph{mutually Grove-Shiohama critical} if $q$ is a Grove-Shiohama critical point of $d_p$ and $p$ is a Grove-Shiohama critical point of $d_q$. 

\begin{proposition}\label{prop}
Let $M$ be a complete Riemannian manifold with sectional curvature $k \geq 1$ and injectivity radius $> \pi/2$. If a pair $p,q \in M$ are mutually Grove-Shiohama critical, then neither $d_p$ nor $d_q$ admits additional Grove-Shiohama critical points of distance. 
\end{proposition}

\begin{proof}
Let $p$ and $q$ be mutually Grove-Shiohama critical and assume by contradiction that $x \neq p,q \in M$ is a Grove-Shiohama critical point of $d_p$. Let $\gamma_2$ be a minimal geodesic from $q$ to $x$. We know there exists a minimal geodesic $\gamma_0$ from $x$ to $p$ such that $\measuredangle{(-\dot{\gamma}_2(d(q,x)),\dot{\gamma}_0(0))} \leq \pi/2$. Since $p$ and $q$ are mutually critical there exist minimal geodesics $\gamma_1$ and $\tilde{\gamma}_1$ from $p$ to $q$ such that $\measuredangle{(\dot{\gamma}_1(0),\dot{\gamma}_0(d(p,x)))} \leq \pi/2$ and $\measuredangle{(-\dot{\tilde{\gamma}}_1(d(p,q)),\dot{\gamma}_2(0))} \leq \pi/2$. Now apply the triangle version of Toponogov's theorem to both $\{\gamma_0, \gamma_1, \gamma_2 \}$ and $\{\gamma_0, \tilde{\gamma}_1, \gamma_2 \}$ yielding comparison triangles in the unit 2-sphere. Because triangles in the sphere are determined up to congruence by side lengths we get a unique comparison triangle, each of whose angles is $\leq \pi/2$. This implies that the comparison triangle is completely contained in an octant of the sphere, hence has side lengths $\leq \pi/2$. This is a contradiction, as $q$ being a Grove-Shiohama critical point of $d_p \colon M \to \mathbb{R}$ implies that $d(p,q) \geq injrad(M) > \pi/2$. 
\end{proof}

We now provide short proofs of Theorem~\ref{main} and Corollary~\ref{flower}.

\begin{proof}[Proof of Theorem~\ref{main}]
Assume that $M$ admits a minimizing geodesic net modeled on such a graph $\Gamma$. Note that the assumed simple closed geodesic must contain the central vertex $p$ for the net. The minimizing condition on the geodesic net implies that the two halves of the simple closed geodesic from $p$ to the halfway point $q$ (also a vertex of the net) both realize distance. Thus this pair is mutually Grove-Shiohama critical. 

By the no cycles of length greater than 2 condition on $\Gamma$, any other vertex $x$ is adjacent only to $p$. If this vertex $x$ were balanced (stationary) then it would be a Grove-Shiohama critical point of $d_p \colon M \to \mathbb{R}$. Thus by Proposition~\ref{prop} the only possible balanced vertices are $p$ and $q$; note that neither of these vertices need be balanced. 
\end{proof}

\begin{proof}[Proof of Corollary~\ref{flower}]
First note that we can model any geodesic flower on a graph $\Gamma$ from Theorem~\ref{main}. In such a model the central vertex together with each of the halfway points along the petals will be balanced vertices. Moreover, these halfway points are all Grove-Shiohama critical points of the central vertex. The result then follows from Theorem~\ref{main}.
\end{proof}

Theorem~\ref{bound} follows from this lemma due to Gromov \cite[Corollary 2.9]{Che}.

\begin{lemma}[Gromov]\label{Gromov} Let $M$ be a complete non-compact Riemannian manifold with sectional curvature $k \geq 0$. Then for every $p \in M$ the distance function $d_p \colon M \to \mathbb{R}$ has no critical points outside of some ball $B(p,R_p)$. In particular, $M$ is homeomorphic to the interior of a compact manifold with boundary. 
\end{lemma}

\begin{proof}[Proof of Theorem~\ref{bound}]
Consider a minimizing geodesic flower with $m$ petals and vertex $p$. The halfway point along each petal is a Grove-Shiohama critical point of $d_p \colon M \to \mathbb{R}$, and by Lemma~\ref{Gromov} must lie in $B(p,R_p)$. We conclude that the total length of the minimizing geodesic flower is bounded above by $(2m)R_p$.
\end{proof}

See also \cite{ade2} where minimizing closed geodesics are studied via Grove-Shiohama critical points of distance, a version of Proposition~\ref{prop}, and Lemma~\ref{Gromov}.


\end{document}